\documentclass[runningheads]{llncs}
\usepackage[T1]{fontenc}
\usepackage{graphicx}
\usepackage{amsmath}
\usepackage{amssymb}
\usepackage[applemac]{inputenc}
\usepackage{latexsym}
\usepackage{graphicx}
\usepackage{color}
\usepackage[german, frenchb, english]{babel}
\usepackage{algorithm2e}
\usepackage[colorlinks]{hyperref}
\usepackage{url}
\usepackage{float}

\newcommand{\ad}{{\rm ad}}

\newcommand{\Tr}{{\rm Tr}\,}

\begin{document}
\title{Hyperk\"ahler Marriage of the two sphere with the hyperbolic space}
%
%
\author{Alice Barbora Tumpach\inst{1, 2}\orcidID{0000-0002-7771-6758} }

\authorrunning{Alice Barbora Tumpach}
%
\institute{Wolfgang Pauli Institut, Oskar-Morgensternplatz 1, 1090 Vienna, Austria\\
 \and
Laboratoire Painlev\'e, Lille University, 59650 Villeneuve d'Ascq, France
\email{alice-barbora.tumpach@univ-lille.fr}\\
 }
\maketitle              
\begin{abstract}The Eguchi-Hanson metric is a natural metric on the total space of the
cotangent bundle $T^*\mathbb{CP}(1)$ of the complex projective line $\mathbb{CP}(1) \simeq \mathbb{S}^2$, which extends the Fubini-Study metric of $\mathbb{CP}(1)$. By virtue of the Mostow decomposition theorem, $T^*\mathbb{CP}(1)$ is isomorphic, as $SU(2)$-equivariant fiber bundle over $\mathbb{CP}(1)$, to a complex (co-)adjoint orbit of $SL(2, \mathbb{C})$. In fact, this  complex (co-)adjoint  orbit is fibered over $\mathbb{CP}(1)\simeq \mathbb{S}^2$ with each fiber isomorphic to the hyperbolic disc $\mathbb{H}^2$. 
In this paper, we are interested in the complex structure inherited on the hyperbolic disc $\mathbb{H}^2$ by the hyperk\"ahler extension of the $2$-sphere. Contrary to what is generally believed, we show that it differs from the natural complex structure of $\mathbb{H}^2\subset \mathbb{C}$ inherited from its embedding in $\mathbb{C}$. In other words, the embedding of $\mathbb{H}^2$ with its Hermitian-symmetric structure into the hyperk\"ahler manifold $T^*\mathbb{CP}(1)$ is not holomorphic.

\end{abstract}
\section{Introduction}

The states of quantum systems are described by wave functions which correspond to probability densities or elements of the complex projective space of a Hilbert space. In classical probability theory, the Fisher information metric can be used to define a   distance between probability densities. The link between the Fubini-Study metric and the Fisher Quantum Information metric tensor is given for instance in \cite{Marmo}. The complex projective space $\mathbb{CP}(1)$ with its Fubini-Study K\"ahler metric is a toy example for Quantum mechanics. 
In this paper, we consider the hyperk\"ahler extension of the complex projective space $\mathbb{CP}(1)$, which can be identified either with $T^*\mathbb{CP}(1)$ or with a complex coadjoint orbit of $SL(2, \mathbb{C})$, and we study the complex structure induced by it on the hyperbolic disc $\mathbb{H}^2$.

\subsection{Definition of hyperk\"ahler metrics}

A hyperk\"ahler manifold can be defined as  a smooth manifold $\mathcal{M}$ endowed with
\begin{itemize}
\item[$\bullet$] one Riemannian metric $\mathrm{g}$,
\item[$\bullet$] three (formally integrable) complex structures $I_1$, $I_2$, $I_3$ satisfying the same identities as the quaternions, in particular  $I_1 I_2 I_3 = -1$,
\item[$\bullet$] three symplectic forms $\omega_1$, $\omega_2$, $\omega_3$ linked to the Riemannian metric $\mathrm{g}$ and the complex structures $I_i$ by  $\omega_i(X, Y) = g(X, I_i Y) $. 
\end{itemize}
Note that, in this setting, $\mathcal{M}$ endow with $I_1$ and $\omega_\mathbb{C}:= \omega_2 + i\omega_3$ is a  complex symplectic manifold. In fact, for any $(a,b,c)\in\mathbb{S}^2$, $I_{(a,b,c)} = a I_1 + b I_2 + c I_3$ is a complex structure on $\mathcal{M}$. Consequently, any hyperk\"ahler manifold admits a family of complex symplectic structures indexed by the two-sphere of complex structures defined by $I_1, I_2, I_3$. Any hyperk\"ahler manifold has a real dimension multiple of $4$.

\subsection{Hyperk\"ahler metrics on cotangent spaces}
The first non-flat example of hyperk\"ahler metric goes back to 1978 with the work of Eguchi and Hanson \cite{Eguchi_Hanson} who constructed a hyperk\"ahler metric, now called Eguchi-Hanson metric, on the cotangent space  of the complex projective line. Note that the real dimension of $T^*\mathbb{CP}(1)$ is 4, which is the minimal dimension for a hyperk\"ahler manifold. The construction of Eguchi and Hanson was then generalized by Calabi \cite{Calabi} to the cotangent space of the complex projective space of any finite dimension, and in \cite{Tum3} to the complex projective space of an infinite-dimensional Hilbert space.  Existence of hyperk\"ahler metrics, for instance on K3 surfaces, can be deduced from Calabi-Yau theorem \cite{Yau}.  However explicit formulas for hyperk\"ahler metrics are rare. As far as we know, none of the hyperk\"ahler metrics known on compact manifolds are explicit.  In the particular case of cotangent bundles of Hermitian-symmetric spaces,  Biquard and Gauduchon \cite{BG1} gave an explicit formula for the hyperk\"ahler metric, generalizing the case of $T^*\mathbb{CP}(n)$. 
This construction  has been used recently to compute  the Gromov width and the Hofer--Zehnder capacity  of  magnetically twisted tangent bundles of Hermitian symmetric spaces \cite{Bimmermann}. For the computation of Hofer--Zehnder capacities on other types of manifolds see \cite{Bimmermann2,Bimmermann3,Bimmermann_Maier}. Note that for Hermitian-symmetric space of non-compact type, the open set on which the hyperk\"ahler metric is defined is linked to Ma\~{n}\'e critical value \cite{Bimmermann}.

\subsection{Isomorphism between cotangent spaces and complex coadjoint orbits}
It is a non trivial fact that cotangent spaces of Hermitian-symmetric spaces of compact type are actually isomorphic to complex coadjoint orbits. This is a consequence of Mostow decomposition theorem \cite{Mostow,Tum1,Tum_Lar,Barbaresco}, which Biquard and Gauduchon \cite{BG2,BG3} used to derive further formulas for the hyperk\"ahler structures mentioned above. 
In particular, their approach illustrates that the natural complex symplectic structure of the cotangent bundle of a Hermitian-symmetric space of compact type and the Kirillov-Kostant-Souriau complex symplectic form of the corresponding complex coadjoint orbit are in fact compatible and form a  hyperk\"ahler metric. The expression of this hyperk\"ahler metric from the point of view of the complex coadjoint orbit will be recalled in section~\ref{hyperkahler}. It uses the fact that  the complex coadjoint orbit is fibered over the Hermitian-symmetric space of compact type, each fiber being isomorphic to the dual Hermitian-symmetric space of non-compact type.  This fibration exists for more general orbits than the Hermitian-symmetric ones, as was shown in \cite[section 2.1]{Tum1} (see section~\ref{section_fibre} below). However, as far as we know, this fibration was not used to understand hyperk\"ahler structures on  more general coadjoint orbits. In the Hermitian-symmetric case, containing in particular Grassmannians and Siegel domains, the constructions of hyperk\"ahler metrics mentioned above have been generalized to the infinite-dimensional setting  in \cite{Tum2,Tum3,Tum4}. 

\subsection{Fibre bundle structure of complex coadjoint orbits}\label{section_fibre}
Let $G$ be a semi-simple connected compact Lie group with Lie algebra $\mathfrak{g}$, $G^{\mathbb{C}}$ the connected Lie group with Lie algebra $\mathfrak{g}^\mathbb{C}:= \mathfrak{g}\oplus i \mathfrak{g}$. Using the Killing form, one can identify adjoint and coadjoint orbits of $G^\mathbb{C}$.
Let $x$ be an element in $\mathfrak{g}$. Denote by $\mathcal{O}_x$ the adjoint orbit of $x$ under $G$ and by $\mathcal{O}^{\mathbb{C}}_x$ the adjoint orbit of $x$ under $G^\mathbb{C}$. Let $p : T\mathcal{O}_x \twoheadrightarrow \mathcal{O}_x$ be the canonical projection of the tangent space $T\mathcal{O}_x$ onto the compact orbit $\mathcal{O}_x$. We recall the following theorem, consequence of Mostow decomposition theorem of $G^{\mathbb{C}}$, \cite[Theorem 2.1]{Tum1}:

\begin{theorem}[\cite{Tum1}]\label{theorem_orbit}
 There exists a $G$-equivariant projection $\pi: \mathcal{O}^\mathbb{C}_x \twoheadrightarrow \mathcal{O}_x$ and a $G$-equivariant diffeomorphism $\Phi: T\mathcal{O}_x \rightarrow \mathcal{O}_x^{\mathbb{C}}$ which commutes with the projections $p$ and~$\pi$.
\end{theorem}
Using the explicit Mostow decomposition of $SL(2, \mathbb{C})$ given in section~\ref{Mostow}, we give the concrete illustration of previous theorem  in the case where $\mathcal{O}_x \simeq \mathbb{CP}(1)$, see Theorem~\ref{cp}.

\subsection{Hyperk\"ahler structures on complexifications of Hermitian-symmetric orbits of compact type}\label{hyperkahler}
The following theorem goes back to \cite{BG1,BG2,BG3} and has been generalized to the case of affine coadjoint orbits of $L^*$-groups in \cite{Tum3,Tum4}. In the setting of previous section, consider a Hermitian-symmetric coadjoint orbit of compact type $\mathcal{O}_x$, $x\in \mathfrak{g}$. This means that the Lie algebra of $\mathfrak{g}$ decomposes into $\mathfrak{g} = \mathfrak{t}_x \oplus \mathfrak{m}_x$, with the following commutation relations $[\mathfrak{t}_x, \mathfrak{m}_x]\subset \mathfrak{m}_x$ and $[\mathfrak{m}_x, \mathfrak{m}_x]\subset \mathfrak{t}_x$, where $\mathfrak{t}_x$ is the stabilizer of $x\in\mathfrak{g}$ for the adjoint action and $\mathfrak{m}_x$ its orthogonal complement in $\mathfrak{g}$ with respect to the Killing form.
Note that the operator $\textrm{ad}_x:=[x, \cdot]$ is anti-symmetric with respect to the Killing form, hence $-\textrm{ad}^2_x$ is symmetric and diagonalizes in an orthonormal basis with respect to the Killing from. Since for $t \in \mathfrak{t}_x$, $\textrm{ad}_t:=[t, \cdot]$ and $\textrm{ad}_x$ commute, the eigenspaces of $-\textrm{ad}^2_x$ are stable by the adjoint action of the isotropy $\mathfrak{t}_x$.
In what follows, we will suppose that $\mathcal{O}_x$ is irreducible, i.e. the isotropy acts in an irreducible way on $T_x\mathcal{O}_x$. 
In this case, there exists a constant $\alpha$ such that $\textrm{ad}_x^2 = -\alpha^2$ on $\mathfrak{m}_x$. In particular, the complex structure of the Hermitian-symmetric space $\mathcal{O}_x$ is given at $T_x\mathcal{O}_x$ by the endomorphism $I_x := \frac{1}{\alpha}\textrm{ad}_x$.

Let $p : T\mathcal{O}_x \twoheadrightarrow \mathcal{O}_x$ be the canonical projection of the tangent space $T\mathcal{O}_x$ onto the compact orbit $\mathcal{O}_x$ and $\pi: \mathcal{O}^{\mathbb{C}}_x \rightarrow \mathcal{O}_x$ the $G$-equivariant projection from the complex orbit onto the compact one given in Theorem~\ref{theorem_orbit}. Endow $\mathfrak{g}^\mathbb{C}$ with  the hermitian scalar product defined by 
\begin{equation}\label{pairing}
\langle \mathfrak{a}, \mathfrak{b} \rangle = \kappa(\mathfrak{a}^*, \mathfrak{b}),~\mathfrak{a}, \mathfrak{b}\in\mathfrak{g}^{\mathbb{C}},
\end{equation}
where $\mathfrak{a}^* = -\mathfrak{a}$ if $\mathfrak{a}\in\mathfrak{g}$ and $\mathfrak{a}^* = \mathfrak{a}$ if $\mathfrak{a}\in i\mathfrak{g}$, and $ \kappa(\cdot,\cdot)$ denotes the Killing form.
We will need to distinguish an element $\mathfrak{c} \in \mathfrak{g}^\mathbb{C}$ and the vector field $X^\mathfrak{c}$ that $\mathfrak{c}$ generates on $\mathcal{O}^\mathbb{C}_x$ by infinitesimal adjoint action: $X^\mathfrak{c}(y) = [\mathfrak{c}, y]$.

\begin{theorem}[\cite{BG1,BG2,BG3}]\label{hyp}
\begin{enumerate}
\item The complex adjoint orbit $\mathcal{O}_{x}^{\mathbb{C}}$ admits a
$G$-invariant hyperk\"ahler structure compatible with the complex
symplectic form $\omega_{\mathbb{C}}$ of Kirillov-Kostant-Souriau and
extending the natural K\"ahler structure of the
Hermitian-symmetric orbit of compact type
$\mathcal{O}_{x}$. 
\item The K\"ahler form $\omega_{1}$ associated with
the natural complex structure  $I_1 := i$ of the complex orbit $\mathcal{O}^{\mathbb{C}}_{x}$ is given by
$\omega_{1} = dd^{c} K$, where the potential $K$ has the following
expression
\begin{equation}\label{potentiel}
K(y) = \alpha \Re \langle y, \pi(y) \rangle, \textrm{where }y\in \mathcal{O}_{x}^{\mathbb{C}}.
\end{equation} 
\item The explicit expression
of the symplectic form $\omega_{1}$ is the following
$$
\begin{array}{l}
\omega_{1}(X^{\mathfrak{c}+ i \mathfrak{c}'}, X^\mathfrak{d + i
\mathfrak{d}'}) = \alpha \Im \left( \langle X^{i\mathfrak{c}'},
\pi_{*}(X^{\mathfrak{d}})\rangle - \langle X^{i\mathfrak{d}'},
\pi_{*}(X^{\mathfrak{c}})\rangle \right).
\end{array}
$$
\item The Riemannian metric is given by
$$
\begin{array}{l}
 \textrm{g}(X^{\mathfrak{c}}, X^{\mathfrak{d}})  =
\textrm{g}(X^{i\mathfrak{c}}, X^{i\mathfrak{d}}) = c \Re \langle
X^{\mathfrak{c}}(y), X^{\mathfrak{d}}(\pi(y)) \rangle,~~~~~~~~
\textrm{g}(X^{\mathfrak{c}}, X^{i\mathfrak{d}}) = 0,
\end{array}
$$
where $\mathfrak{c}$, $\mathfrak{c}'$, $\mathfrak{d}$ and
$\mathfrak{d}'$ belong to $\mathfrak{m}_{\pi(y)}$. 
\item The complex
structure $I_2$, extending the complex structure $I$ of the Hermitian-symmetric space of compact  type $\mathcal{O}_x$,  is given at $y\in \pi^{-1}(x)$ by
\begin{equation}\label{I2}
I_{2}X^{\mathfrak{d}} = X^{\frac{1}{\alpha}\textrm{ad}_x{\mathfrak{d}}},~~~~~~~~ I_{2}X^{i\mathfrak{d}} =
-X^{\frac{1}{\alpha}\textrm{ad}_x{i\mathfrak{d}}},
\end{equation}
where $\mathfrak{c}$ and $\mathfrak{d}$ belong to
$\mathfrak{m}_x$ and where $\alpha$ is such that $\textrm{ad}_x^2 = -\alpha^2$ on $\mathfrak{m}_x$.
\end{enumerate}
\end{theorem}

\subsection{Contributions}
 
  \subsubsection{Summary:}
 Mostow decomposition theorem leads to an identification of $T^*\mathbb{CP}(1)$ with the space $\mathcal{O}_x^{\mathbb{C}}$ consisting of pairs $(\ell_1, \ell_2)$ of complex lines in $\mathbb{C}^2$ such that $\ell_1 \oplus \ell_2 = \mathbb{C}^2$. This later space is a coadjoint orbit of ${SL}(2, \mathbb{C})$ and is fibered over $\mathbb{CP}(1)$  with fibers isomorphic with the hyperbolic disc $\mathbb{H}^2$. The hyperk\"ahler structure on $\mathcal{O}_x^{\mathbb{C}}$ defined in Theorem~\ref{hyp} gives a complex structure $I_2$ which preserves $\mathbb{H}^2$ but differs from the natural complex structure on $\mathbb{H}^2$ induced by its embedding in $\mathbb{C}$.
 
 \subsubsection{Key points:}
  \begin{itemize}  
  \item[$\bullet$] In Proposition~\ref{non-compact}, an explicit parameterization of the orbit of non-compact type $\mathcal{O}_x^{n.c.}$, embedded in the complex coadjoint orbit $\mathcal{O}_x^{\mathbb{C}}$, is given, as well as a diffeomorphism with $\mathbb{H}^2$.
  \item[$\bullet$] In Proposition~\ref{Mostow_SL},  Mostow decomposition of $SL(2, \mathbb{C})$ is computed explicitly.
  \item[$\bullet$] In Theorem~\ref{cp}, the fibration of $\mathcal{O}_x^{\mathbb{C}}$ over $\mathcal{O}_x$ is given, based on Mostow decomposition of $SL(2, \mathbb{C})$. This allows to define the  $SU(2)$-equivariant diffeomorphism $\Phi$ between $\mathcal{O}_x^{\mathbb{C}}$ and $T\mathbb{CP}(1)\simeq T^*\mathbb{CP}(1)$.
  \item[$\bullet$] In Theorem~\ref{non-holo}, the complex structure induced on $\mathbb{H}^2$ by the complex structure $I_2$ of $\mathcal{O}_x^{\mathbb{C}}$ defined by \eqref{I2} is computed. It differs from the natural complex structure of $\mathbb{H}^2$ as open subset of $\mathbb{C}$. 
  \item[$\bullet$] A visualization of the complex structure induced by $I_2$ on $\mathbb{H}_2$ is given in Figure~\ref{fig}.
  \end{itemize}

\section{The $2$-sphere and the hyperbolic disc are living in the same coadjoint orbit of $SL(2, \mathbb{C})$}

\subsection{Presentation of the protagonists and their relatives}

In the following, we will identify the $2$-sphere with the complex projective line $\mathbb{CP}^1$  endowed with its K\"ahler metric known as the Fubini-Study metric.  The corresponding Riemannian metric is 4 times the Riemannian metric of the round sphere of radius 1,
and the complex structure is the one of the Riemannian sphere $\mathbb{S}^2 \simeq \mathbb{C}\cup\{\infty\} $.
The hyperbolic Poincare disc will be denoted by $\mathbb{H}^2 = \{ z\in \mathbb{C}, |z|<1\}$.
For the remainder of the paper, we will restrict our attention to the Lie group 
\[
G = SU(2) := \left\{ \left(\begin{array}{cc} a & -\bar{b}\\b & \bar{a} \end{array}\right), |a|^2 + |b|^2 = 1, a, b\in\mathbb{C}\right\},
\]
with Lie algebra
$
\mathfrak{g} = \mathfrak{su}(2) := \left\{  \left(\begin{array}{cc} i\alpha & -\bar{\beta}\\\beta & -i\alpha \end{array}\right),\alpha\in\mathbb{R}, \beta\in\mathbb{C}\right\}.
$
It is a maximal compact subgroup of the complex Lie group
\[
G^{\mathbb{C}}= SL(2, \mathbb{C}) := \left\{ \left(\begin{array}{cc} a & b\\c & d \end{array}\right), ad-bc = 1, a,b,c,d \in\mathbb{C} \right\},
\] 
with Lie algebra 
$
\mathfrak{g}^\mathbb{C} = \mathfrak{sl}(2, \mathbb{C}) :=  \left\{ \left(\begin{array}{cc} a & b\\c & -a \end{array}\right),  a,b,c \in\mathbb{C} \right\}.
$

\subsection{Compact (co-)adjoint orbit $\mathcal{O}_x$, complex (co-)adjoint orbit $\mathcal{O}^\mathbb{C}_x$ and non-compact (co-)adjoint orbit $\mathcal{O}^{n.c.}_x$}
Consider the element
$$x := \left(\begin{array}{cc} i & 0\\0 & -i \end{array}\right)\in \mathfrak{su}(2).$$
Let $G$ and $G^{\mathbb{C}}$ act by adjoint action on their Lie algebras, and denote by $\mathcal{O}_{x}$ the compact adjoint orbit of $x$ under $G$, and $\mathcal{O}_x^{\mathbb{C}}$ the complex adjoint orbit of $x$ under $G^\mathbb{C}$. One has:
\[
\begin{array}{cll}
\mathcal{O}_x & \simeq & \mathbb{CP}^1 \\
\left(\begin{array}{cc} a & -\bar{c}\\c & \bar{a} \end{array}\right) \left(\begin{array}{cc} i & 0\\0 & -i \end{array}\right) \left(\begin{array}{cc} a & -\bar{c}\\c & \bar{a} \end{array}\right)^{-1} & \mapsto &  \textrm{Eig}(i)  = \mathbb{C}\left(\begin{array}{c}a \\c\end{array}\right),
\end{array}
\]
and similarly
\[
\begin{array}{cll}
\mathcal{O}^\mathbb{C}_x & \simeq & \left\{ (\ell_1, \ell_2)\in \mathbb{CP}^1\times \mathbb{CP}^1, \ell_1 \oplus \ell_2 = \mathbb{C}^2\right\}\\
\left(\begin{array}{cc} a & b\\c & d \end{array}\right) \left(\begin{array}{cc} i & 0\\0 & -i \end{array}\right) \left(\begin{array}{cc} a & b\\c & d \end{array}\right)^{-1} & \mapsto &  \left(\textrm{Eig}(i) = \mathbb{C}\left(\begin{array}{c}a \\c\end{array}\right), \textrm{Eig}(-i) = \mathbb{C}\left(\begin{array}{c}b \\d\end{array}\right) \right).
\end{array}
\]
The stabilizers of $x$ for the adjoint actions of $G$ and $G^{\mathbb{C}}$ are respectively:
\[
\begin{array}{l}
 \operatorname{Stab}_{SU(2)}(x) =  U(1) =  \left\{ \left(\begin{array}{cc} e^{i\theta} & 0\\0 & e^{-i\theta} \end{array}\right), \theta \in \mathbb{R}\right\},\\
\operatorname{Stab}_{SL(2, \mathbb{C})}(x) = \mathbb{C}^* = \left\{ \left(\begin{array}{cc} a & 0\\0 & \frac{1}{a} \end{array}\right), a\in \mathbb{C}^*\right\}.\\
\end{array}
\]
Denote by $\mathfrak{t}$ the Lie algebra of $ \operatorname{Stab}_{SU(2)}(x)$. Then the Lie algebra of $\operatorname{Stab}_{SL(2, \mathbb{C})}(x) $ is $\mathfrak{t}^\mathbb{C} = \mathfrak{t} \oplus i \mathfrak{t}$:
$$
\begin{array}{l}
\mathfrak{t} = \left\{ \left(\begin{array}{cc} {i\theta} & 0\\0 & {-i\theta} \end{array}\right), \theta \in \mathbb{R}\right\},\\
\mathfrak{t}^\mathbb{C} =  \mathfrak{t} \oplus i \mathfrak{t} = \left\{ \left(\begin{array}{cc} a & 0\\0 & -a \end{array}\right), a \in \mathbb{C}\right\}.\\
\end{array}
$$
Let us introduce the orthogonal  $\mathfrak{m}$ of $\mathfrak{t}$ in $\mathfrak{su}(2)$ with respect to the Killing form $K(A, B) := \Tr(\ad(A)\ad(B)) = 4 \Tr AB$:
$$
\mathfrak{m} := \left\{ \left(\begin{array}{cc} 0  & -\bar{c}\\c & 0\end{array}\right), c \in \mathbb{C}\right\}.
$$
Then
$
\mathfrak{g} = \mathfrak{su}(2) = \mathfrak{t}\oplus \mathfrak{m}$
 and 
$\mathfrak{g}^\mathbb{C} = \mathfrak{sl}(2, \mathbb{C}) = \mathfrak{t}\oplus \mathfrak{m} \oplus i \mathfrak{t} \oplus i \mathfrak{m}, 
$
with the commutation relations $[\mathfrak{t}, \mathfrak{m}] \subset \mathfrak{m}$ and $[\mathfrak{m}, \mathfrak{m}] \subset \mathfrak{t}$. Denote by 
$$
\mathfrak{g}^{n.c.} := \mathfrak{su}(1, 1) = \mathfrak{t}\oplus i\mathfrak{m} = \left\{ \left(\begin{array}{cc} i\theta  & \bar{\beta}\\\beta & -i\theta\end{array}\right), \theta\in \mathbb{R}, \beta \in \mathbb{C}\right\}
$$
 the Lie algebra of the non-compact semi-simple Lie group $SU(1,1)$:
 \begin{equation}\label{su11}
 SU(1,1) := \left\{ \left(\begin{array}{cc} a & b\\\bar{b} & \bar{a} \end{array}\right), a, b \in \mathbb{C},  |a|^2 - |b|^2 = 1 \right\}.
 \end{equation}
 Note that 
 $x := \left(\begin{array}{cc} i & 0\\0 & -i \end{array}\right)\in \mathfrak{su}(2)\cap \mathfrak{su}(1, 1) = \mathfrak{t},$
and that the stabilizer of $x$ under the adjoint action of $SU(1,1)$ is $U(1)$.
 It follows that the non-compact (co-)adjoint orbit $\mathcal{O}_x^{n.c.}$ of $x$ under the adjoint action of $SU(1,1)$ is diffeomorphic to the hyberbolic $2$-space presented as the hyperbolic disc $\mathbb{H}^2$. More explicitly, one has
 \begin{proposition} \label{non-compact}
 The non-compact orbit $\mathcal{O}_x^{n.c.}$ is isomorphic to the quotient space $SU(1,1)/U(1)$ hence to  $\mathbb{H}^2$, 
 \[
 \mathcal{O}_x^{n.c.} \simeq SU(1,1)/U(1) \simeq \mathbb{H}^2
 \]
and can be parameterized by $\beta\in \mathbb{C}$ as follows
 \[
  \begin{array}{cccll}
 \mathcal{O}_x^{n.c.} & \longrightarrow &  SU(1,1)/U(1) & \longrightarrow & \mathbb{H}^2 \\
 \textrm{Ad}\left(\begin{smallmatrix}\cosh |\beta| & \frac{\beta}{|\beta|} \sinh |\beta| \\ \frac{\bar{\beta}}{|\beta|} \sinh |\beta| & \cosh |\beta|\end{smallmatrix}\right)(x) &  \longmapsto &  \left[\left(\begin{smallmatrix}\cosh |\beta| & \frac{\beta}{|\beta|} \sinh |\beta| \\ \frac{\bar{\beta}}{|\beta|} \sinh |\beta| & \cosh |\beta|\end{smallmatrix}\right)\right]_{U(1)}  & \longmapsto &  \frac{\beta}{|\beta|} \tanh |\beta|,
 \end{array}
 \]
The inverse diffeomorphisms are given by
 \[
 \begin{array}{lccll}
 \mathbb{H}^2 & \longrightarrow &  SU(1,1)/U(1) & \longrightarrow &  \mathcal{O}_x^{n.c.}\\
 r e^{i\theta} &  \longmapsto &  \left[\left(\begin{smallmatrix}\cosh \frac{1}{2} \ln \left(\frac{1+r}{1-r}\right)& e^{i\theta} \sinh \frac{1}{2} \ln \left(\frac{1+r}{1-r}\right)\\ e^{-i\theta}\sinh \frac{1}{2} \ln \left(\frac{1+r}{1-r}\right) & \cosh \frac{1}{2} \ln \left(\frac{1+r}{1-r}\right) \end{smallmatrix}\right)\right]_{U(1)}  & \longmapsto &  \textrm{Ad}\left(\begin{smallmatrix}\cosh \frac{1}{2} \ln \left(\frac{1+r}{1-r}\right) & e^{i\theta} \sinh \frac{1}{2} \ln \left(\frac{1+r}{1-r}\right)\\ e^{-i\theta}\sinh \frac{1}{2} \ln \left(\frac{1+r}{1-r}\right) & \cosh \frac{1}{2} \ln \left(\frac{1+r}{1-r}\right)\end{smallmatrix}\right)(x).
 \end{array}
 \]
  \end{proposition}
 
 \begin{proof}
 We omit the computation which can be easily deduced from the polar decomposition $SU(1,1) = \exp i\mathfrak{m}\times U(1)$. 
 \end{proof}

\subsection{Explicit Mostow Decomposition of $SL(2, \mathbb{C})$}\label{Mostow}

Mostow decomposition Theorem \cite{Mostow,Tum1,Tum_Lar} gives a diffeomorphism
\begin{equation}\label{sl2}
SL(2, \mathbb{C})  = SU(2) \times \exp i\mathfrak{m} \times \exp i\mathfrak{t},
\end{equation}
where 
$
i\mathfrak{t} :=  \left\{  \left(\begin{array}{cc} \alpha & 0\\0 & -\alpha \end{array}\right), \alpha\in\mathbb{R} \right\}
$ and  
 $
i\mathfrak{m}  = \left\{  \left(\begin{array}{cc} 0 & \bar{\beta} \\ \beta & 0 \end{array}\right),  \beta\in\mathbb{C}\right\}.
$
\begin{proposition}\label{Mostow_SL}
Mostow decomposition 
$
SL(2, \mathbb{C})  = SU(2) \times \exp i\mathfrak{m} \times \exp i\mathfrak{t},
$
maps a matrix $\left(\begin{array}{cc} a & b\\c & d \end{array}\right)\in SL(2, \mathbb{C})$ to the product $k\cdot f\cdot e$, where
\[
\begin{array}{l}
e :=   \left(\begin{array}{cc} \alpha & 0\\0 & 1/\alpha \end{array}\right)\in\exp i\mathfrak{t}, \textrm{ with } \alpha := \sqrt{\frac{|a|^2 + |c|^2}{|b|^2 + |d|^2}},
\\
f := \left(\begin{array}{cc} \cosh |\beta| & \frac{\beta}{|\beta|}\sinh |\beta| \\\frac{\bar{\beta}}{|\beta|}\sinh|\beta| & \cosh |\beta| \end{array}\right)\in\exp i\mathfrak{m}, 
\end{array}
\]
with
$$
\beta := \frac{\bar{a}b + \bar{c} d}{|\bar{a}b + \bar{c} d|} \frac{\ln\left(\sqrt{(|a|^2 + |c|^2)(|b|^2 + |d|^2)} + |\bar{a}b + \bar{c} d|\right)}{2}
$$
and
\begin{equation}
k := \left(\begin{array}{cc} a & b\\c & d \end{array}\right) e^{-1} f^{-1}.
\end{equation}
\end{proposition}
\begin{proof}
Indeed, one has 
\begin{equation}\label{rhs}
e f^2 e = \left(\begin{array}{cc} a & b\\c & d \end{array}\right)^*\left(\begin{array}{cc} a & b\\c & d \end{array}\right) = \left(\begin{array}{cc} |a|^2 + |c|^2 & \bar{a}b+\bar{c} d\\\bar{b} a + \bar{d} c & |b|^2 + |d|^2\end{array}\right),
\end{equation}
and 
$$
e f^2 e = \left(\begin{array}{cc} \alpha^2 \cosh 2 |\beta|  & \frac{\beta}{|\beta|} \sinh 2 |\beta| \\ \frac{\bar{\beta}}{|\beta|} \sinh 2|\beta|& \frac{1}{\alpha^2}\cosh  2|\beta| \end{array}\right),
$$
with 
$
\cosh 2|\beta| = \sqrt{(|a|^2 + |c|^2)(|b|^2 + |d|^2)}
$
and
$
\sinh 2|\beta| = |\bar{a}b + \bar{c}d|.
$
Note that since $\det \left(\begin{array}{cc} a & b\\c & d \end{array}\right) = 1$,  the determinant of the right hand side of \eqref{rhs} is $1$. Hence,  one verifies that $ \cosh^2 2|\beta| - \sinh^2 2|\beta| = (|a|^2 + |c|^2)(|b|^2 + |d|^2) - |\bar{a}b + \bar{c}d|^2 = 1$.
\end{proof}

\subsection{Diffeomorphism between $T\mathbb{CP}^1$ and $ \mathcal{O}^\mathbb{C}_x$}

\begin{theorem}\label{cp}
\begin{enumerate}
\item 
There is a $SU(2)$-equivariant fibration $\pi: \mathcal{O}^\mathbb{C}_x\rightarrow \mathcal{O}_x$ of the complex (co-)adjoint orbit 
\[
\mathcal{O}^\mathbb{C}_x  \simeq  \left\{ (\ell_1, \ell_2)\in \mathbb{CP}^1\times \mathbb{CP}^1, \ell_1 \oplus \ell_2 = \mathbb{C}^2\right\}
\]
over the compact (co-)adjoint orbit $\mathcal{O}_x\simeq \mathbb{CP}(1)\simeq \mathbb{S}^2$ given by
\[
 \begin{array}{cccc}
 \pi: & \mathcal{O}^\mathbb{C}_x &\longrightarrow& \mathcal{O}_x\\
 & \textrm{Ad}(k\cdot f\cdot e)(x) &\longmapsto& \textrm{Ad}(k)(x),
 \end{array}
\]
where $(k, f, e)\in SU(2) \times \exp i\mathfrak{m} \times \exp i\mathfrak{t}$.
In particular, the fiber $\pi^{-1}(x)$ over $x\in\mathcal{O}_x^{\mathbb{C}}$ is $\textrm{Ad}(e^{i \mathfrak{m}})(x) = \mathcal{O}_x^{n.c.}\simeq \mathbb{H}^2$. 

\item Moreover the complex (co-)adjoint orbit is isomorphic to the (co-)tangent bundle of $\mathbb{CP}^1$ via the following $SU(2)$-equivariant diffeomorphism
\[
  \begin{array}{cccc}
 \Phi: & T\mathbb{CP}^1 = T\mathcal{O}_x  & \longrightarrow &  \mathcal{O}^\mathbb{C}_x  \\
&  \left(z = g~x~g^{-1}, v = g~\mathfrak{a}~g^{-1}\right) \textrm{ with } g\in SU(2) \textrm{ and }\mathfrak{a}\in\mathfrak{m} & \longmapsto & g~e^{i \mathfrak{a}}~ x~ e^{- i \mathfrak{a}} ~g^{-1}
 \end{array}
 \]
\end{enumerate}
In other words,  $\mathcal{O}_x^{\mathbb{C}}$ is the complexification of the compact orbit $\mathcal{O}_x\simeq \mathbb{CP}^1$ and is isomorphic to the tangent bundle of $\mathbb{CP}^1$. The hyperbolic space $\mathbb{H}^2$ sits inside  $\mathcal{O}_x^{\mathbb{C}}$ as the fiber over $x\in \mathcal{O}_x \simeq \mathbb{CP}^1$ and is the orbit of $x$ under the adjoint action of $SU(1,1)\subset SL(2,\mathbb{C})$.
\end{theorem}

\begin{proof}
The map $\pi$ is well-defined by the uniqueness of Mostow decomposition~\eqref{sl2} 
together with the fact that $\exp i\mathfrak{t}$ stabilizes $x\in \mathcal{O}_x$. The tangent space at $x\in\mathbb{CP}(1)$ is identified with $\mathfrak{m} = \mathfrak{g}/\mathfrak{t}$, and the tangent space at $g x g^{-1}\in\mathbb{CP}(1)$ is identified with $g\mathfrak{m}g^{-1}.$ Note that $\mathfrak{m}$ is stable under the adjoint action of the isotropy group.  The diffeomorphism $\Phi$ maps the fibers of the canonical projection $p: T\mathbb{CP}^1\rightarrow \mathbb{CP}^1$ to the fibers of  $\pi: \mathcal{O}^\mathbb{C}_x\rightarrow \mathcal{O}_x$ and is clearly a diffeomorphism.
\end{proof}

\section{The hyperbolic  disc $\mathbb{H}^2$ inside $T\mathbb{CP}^1 \simeq \mathcal{O}^\mathbb{C}_x $}

\subsection{The complex structure of $\mathbb{CP}^1$ viewed as coadjoint orbit}

As before, 
consider $G = SU(2)$, $G^{\mathbb{C}} = SL(2, \mathbb{C})$ and the adjoint orbits $\mathcal{O}_x = G\cdot x$ and $\mathcal{O}_x^\mathbb{C} = G^\mathbb{C}\cdot x$ of 
$x := \left(\begin{array}{cc} i & 0\\0 & -i \end{array}\right)\in \mathfrak{su}(2).$
Then $\textrm{ad}_x$ acts on 
$
T_x\mathcal{O}_x \simeq \mathfrak{m} := \left\{ \left(\begin{array}{cc} 0  & -\bar{c}\\c & 0\end{array}\right), c \in \mathbb{C}\right\}
$
by
$
\textrm{ad}_x \left(\begin{array}{cc} 0  & -\bar{c}\\c & 0\end{array}\right) = 2 \left(\begin{array}{cc} 0  & -i\bar{c}\\-ic & 0\end{array}\right)
$
hence the constant $\alpha$ appearing in Theorem~\ref{hyp} is $\alpha = \frac{1}{2}$ and the complex structure $I$ of the Hermitian-symmetric orbit $\mathcal{O}_x\simeq \mathbb{CP}(1)$ is given at $T_x\mathcal{O}_x$ by 
$
I X^{\mathfrak{c}} := X^{\frac{1}{2}\textrm{ad}_x\mathfrak{c}}, 
$
$\mathfrak{c}\in\mathfrak{m}$. 
\subsection{The complex structure $I_2$ of the hyperk\"ahler manifold $T\mathbb{CP}^1 \simeq \mathcal{O}^\mathbb{C}_x $}

\begin{theorem}\label{non-holo}
The complex structure $I_2$ defined in \eqref{I2} on the hyperk\"ahler manifold  $\mathcal{O}^\mathbb{C}_x\simeq T^*\mathbb{CP}(1)$ and extending the complex structure of $\mathbb{CP}(1)$, restricts to a complex structure on $\mathbb{H}^2$ which differs from the natural complex structure of $\mathbb{H}^2 = \{ z\in \mathbb{C}, |z|<1\}$ induced by its embedding into $\mathbb{C}$, see Figure~\ref{fig}.
\end{theorem}
\begin{figure}[H]\label{fig}
\centering
\includegraphics[width = 0.5\linewidth]{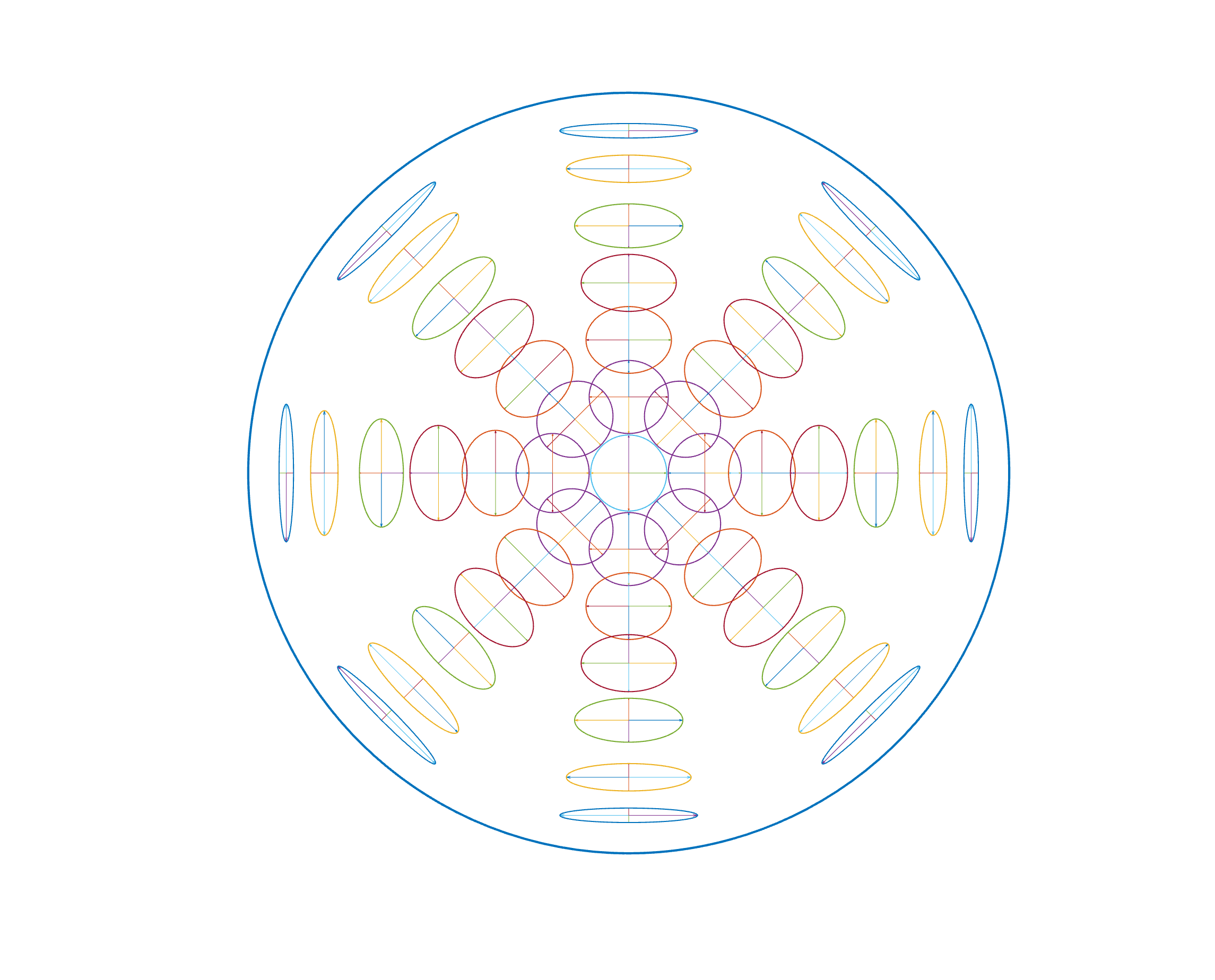}
\caption{Representation of the complex structure induced on $\mathbb{H}^2$ by the complex structure $I_2$ of the hyperk\"ahler coadjoint orbit $\mathcal{O}_x^\mathbb{C}\simeq T^*\mathbb{CP}(1)$. Each ellipse represents the complex structure $I_2$ at the center of the ellipse. The complex structure maps the tangent vector given at the center of each ellipse  by half of one axis  to the tangent vector given by half of the other axis.}
\end{figure}

\begin{proof}
The complex structure $I_2$ on $\mathcal{O}^\mathbb{C}_x$ defined by \eqref{I2} extends the complex structure $I$ of $\mathcal{O}_x = \mathbb{CP}^1$. Moreover $I_2$ restricts to a complex structure on $\pi^{-1}(x) = \mathcal{O}_x^{n.c.}\simeq \mathbb{H}^2$ given at 
$$
y :=  \textrm{Ad}\left(\begin{smallmatrix}\cosh |\beta| & \frac{\beta}{|\beta|} \sinh |\beta| \\ \frac{\bar{\beta}}{|\beta|} \sinh |\beta| & \cosh |\beta|\end{smallmatrix}\right)(x)\in \mathcal{O}_x^{n.c.}
$$
by the endomorphism of the tangent space $T_y\mathcal{O}_x^{n.c.}$ which maps the tangent vector $X^{i\mathfrak{d}}(y) $ generated by $i\mathfrak{d} := \begin{pmatrix}0  & c\\ \bar{c} & 0 \end{pmatrix}$ to the tangent vector $I_2 X^{i\mathfrak{d}}(y) $ given by
\[
\begin{array}{ll}
I_2 X^{i\mathfrak{d}}(y) &= I_2 \left[  \left(\begin{array}{cc} 0  & c\\ \bar{c} & 0\end{array}\right), y \right] = -\left[  \frac{1}{2} \left[\left(\begin{array}{cc} i  & 0\\0 & -i\end{array}\right), \left(\begin{array}{cc} 0  & c\\ \bar{c}& 0\end{array}\right)\right], y \right]
 =  \left[  \left(\begin{array}{cc} 0  & -ic\\ i\bar{c} & 0\end{array}\right), y \right].
\end{array}
\]
Recall that $SU(1,1)$ acts on $\mathbb{H}^2$ by M\"obius transformations:
\[
\begin{pmatrix}a & b\\ \bar{b} & \bar{a} \end{pmatrix}\cdot Z = \frac{a Z + b}{\bar{b} Z + \bar{a}}.
\]
In particular the base point $0 \in \mathbb{H}^2$ is mapped to $\frac{b}{\bar{a}}$. The isotropy of $0\in \mathbb{H}^2$ is $U(1)$, hence the quotient space $SU(1,1)/U(1)$ is identified with the orbit of $0$ by 
\[
\left[\begin{pmatrix}a & b\\ \bar{b} & \bar{a} \end{pmatrix}\right]_{U(1)} \mapsto \frac{b}{\bar{a}}\in\mathbb{H}^2.
\]
Given a tangent vector $X$ in $T_y\mathcal{O}_x^{n.c.}$ written as the velocity vector of a curve of the form  $c(t) = \textrm{Ad}(g(t))(x)$, with $g(t) = \begin{pmatrix}a(t) & b(t)\\ \bar{b(t)} & \bar{a(t)} \end{pmatrix} \in SU(1,1)$, the corresponding tangent vector to the quotient space $SU(1,1)/U(1)$
is the velocity vector of the curve $\left[\begin{pmatrix}a(t) & b(t)\\ \bar{b}(t) & \bar{a}(t) \end{pmatrix}\right]_{U(1)}$, and the corresponding vector in $\mathbb{H}^2$ is the velocity of the curve $\tilde{c}(t) = \frac{b(t)}{\bar{a}(t)}$.
We will consider the following two tangent vectors spanning the tangent space $T_y\mathcal{O}_x^{n.c.}$:
\[
X_1 = \left[  \left(\begin{array}{cc} 0  & i\\ -i & 0\end{array}\right), y \right]
\textrm{ and }
X_2 = \left[  \left(\begin{array}{cc} 0  & 1\\ 1 & 0\end{array}\right), y \right].
\]
Note that $I_2 X_1 = X_2$ and $I_2 X_2 = - X_1$. One has
\[
\begin{array}{ll}
X_1 & = \frac{d}{dt}_{|t = 0} \textrm{Ad}\left(\exp  t\left(\begin{array}{cc} 0  & i\\ -i & 0\end{array}\right)\right)(y) \\ 
&= \frac{d}{dt}_{|t = 0} \textrm{Ad}\left(\exp  t\left(\begin{array}{cc} 0  & i\\ -i & 0\end{array}\right) 
\cdot \exp  \left(\begin{array}{cc} 0  & \beta\\ \bar{\beta}  & 0\end{array}\right)\right)(x)\\
& = \frac{d}{dt}_{|t = 0} \textrm{Ad}\left(\left(\begin{array}{cc}\cosh t & i\sinh t \\ -i \sinh t & \cosh t\end{array}\right)\cdot\left(\begin{array}{cc}\cosh |\beta| & \frac{\beta}{|\beta|} \sinh |\beta| \\ \frac{\bar{\beta}}{|\beta|} \sinh |\beta| & \cosh |\beta|\end{array}\right)\right)(x)\\
& = \frac{d}{dt}_{|t = 0} \textrm{Ad}\left(\begin{array}{cr} \cosh t \cosh |\beta| +\frac{i\bar{\beta}}{|\beta|} \sinh t \sinh |\beta| & \quad \frac{\beta}{|\beta|}\cosh t  \sinh |\beta| + i \sinh t \cosh |\beta|\\   -i  \sinh t \cosh |\beta|  + \frac{\bar{\beta}}{|\beta|} \cosh t \sinh |\beta| &  -i \frac{\beta}{|\beta|} \sinh t \sinh |\beta| + \cosh t \cosh |\beta| \end{array}\right)(x)\\
\end{array},
\]
where the  matrix appearing in the last line belongs to $SU(1,1)$ defined in \eqref{su11}. Similarly
\[
\begin{array}{ll}
X_2 & = \frac{d}{dt}_{|t = 0} \textrm{Ad}\left(\exp  t\left(\begin{array}{cc} 0  & 1\\ 1 & 0\end{array}\right)\right)(y) \\ 
&= \frac{d}{dt}_{|t = 0} \textrm{Ad}\left(\exp  t\left(\begin{array}{cc} 0  & 1\\ 1 & 0\end{array}\right) 
\cdot \exp  \left(\begin{array}{cc} 0  & \beta\\ \bar{\beta}  & 0\end{array}\right)\right)(x)\\
& = \frac{d}{dt}_{|t = 0} \textrm{Ad}\left(\left(\begin{array}{cc}\cosh t & \sinh t \\  \sinh t & \cosh t\end{array}\right)\cdot\left(\begin{array}{cc}\cosh |\beta| & \frac{\beta}{|\beta|} \sinh |\beta| \\ \frac{\bar{\beta}}{|\beta|} \sinh |\beta| & \cosh |\beta|\end{array}\right)\right)(x)\\
& = \frac{d}{dt}_{|t = 0} \textrm{Ad}\left(\begin{array}{cr} \cosh t \cosh |\beta| +\frac{\bar{\beta}}{|\beta|} \sinh t \sinh |\beta| & \quad \frac{\beta}{|\beta|}\cosh t  \sinh |\beta| +  \sinh t \cosh |\beta|\\     \sinh t \cosh |\beta|  + \frac{\bar{\beta}}{|\beta|} \cosh t \sinh |\beta| &  \frac{\beta}{|\beta|} \sinh t \sinh |\beta| + \cosh t \cosh |\beta| \end{array}\right)(x)\\
\end{array}
\]
It follows that the tangent vector $X_1$ (resp. $X_2$) corresponds to the tangent vector $\tilde{X}_1$ (resp. $\tilde{X}_2$) at $Z = \frac{\beta}{|\beta|}\tanh |\beta| \in \mathbb{H}^2$ given by
\[
\begin{array}{l}
\tilde{X}_1 = \frac{d}{dt}_{|t = 0} \left(\frac{\frac{\beta}{|\beta|}\cosh t  \sinh |\beta| + i \sinh t \cosh |\beta|}{-i \frac{\beta}{|\beta|} \sinh t \sinh |\beta| + \cosh t \cosh |\beta|}\right) = i \frac{\cosh^2|\beta| + \frac{\beta^2}{|\beta|^2} \sinh^2 |\beta|}{\cosh^2|\beta|} = i\left(1 + \frac{\beta^2}{|\beta|^2} \tanh^2 |\beta| \right) = i (1 + Z^2)
\\
\tilde{X}_2 = \frac{d}{dt}_{|t = 0} \left(\frac{\frac{\beta}{|\beta|}\cosh t  \sinh |\beta| +  \sinh t \cosh |\beta|}{\frac{\beta}{|\beta|} \sinh t \sinh |\beta| + \cosh t \cosh |\beta|}\right) =  \frac{\cosh^2|\beta| - \frac{\beta^2}{|\beta|^2} \sinh^2 |\beta|}{\cosh^2|\beta|} = 1 - Z^2
\end{array}
\]
Note that, in the basis $\{\tilde{X}_1, \tilde{X}_2\}$ of $T_Z\mathbb{H}^2$, the complex structure induced by $I_2$ on $\mathbb{H}^2$  reads $\begin{pmatrix} 0 & -1\\ 1 & 0\end{pmatrix}$. 
However $\{\tilde{X}_1, \tilde{X}_2\}$ is not an orthonormal basis for the Euclidean metric on $\mathbb{C}\simeq \mathbb{R}^2$. At $Z = r e^{i\theta}$, the expressions of $\tilde{X}_1$ and $\tilde{X}_2$ in the canonical basis of $\mathbb{R}^2$ are the columns of the following matrix
\[
P = \begin{pmatrix} -r^2 \sin(2\theta) & 1 -r^2 \cos(2\theta)\\ 1 +r^2 \cos(2\theta) & -r^2 \sin(2\theta) \end{pmatrix}
\]
The singular value decomposition of $P$, $P= U S V$, allows to write the complex structure induced by $I_2$ at $T_Z\mathbb{H}^2$ as the endomorphism of $\mathbb{R}^2$ given by
\[
P\begin{pmatrix} 0 & -1\\ 1 & 0\end{pmatrix}P^{-1} = U S \begin{pmatrix} 0 & -1\\ 1 & 0\end{pmatrix} S^{-1} U^{-1} = U \begin{pmatrix} 0 & -\frac{a}{b}\\ \frac{b}{a} & 0\end{pmatrix} U,
\] where $S$ is a diagonal matrix containing the singular values $a, b$ of $P$ ordered in decreasing order, and $U$ and $V$ are orthogonal matrices. Denote by $U_1$ and $U_2$ the columns of $U$. The complex structure at $Z = re^{i\theta}$ maps $U_1$ to $\frac{b}{a} U_2$ and $U_2$ to $-\frac{a}{b} U_1$. In Figure~\ref{fig}, the complex structure  is illustrated at  $Z = re^{i\theta}$ for $r =  \{0, 0.2, 0.35, 0.5, 0.65, 0.8, 0.9\}$ and $\theta = k \frac{2\pi}{8}, k = 0, \dots, 8$. Half of the great axis of each ellipse corresponds to a multiple of $U_1$ and its image by the complex structure is the vector corresponding to half of the small axis of the ellipse and making an angle of $+\frac{\pi}{2}$ with the previous one.
\end{proof}

\begin{credits}
\subsubsection{\ackname} This research is funded by FWF under the grant numbers~I-5015-N and PAT1179524. 
The author would like to thank Milan Niestijl for discussions on the subject at the ``Higher Structures and Field Theory'' conference held at Erwin Schroedinger Institut in Vienna, and for careful reading of the manuscript.
 We acknowledge the excellent working conditions and
     interactions at Erwin Schroedinger Institut, Vienna, during the
     thematic programme ``Infinite-dimensional Geometry: Theory and
     Applications'' where part of this work was completed. 
\end{credits}

%
%
%

\end{document}